\documentclass[11pt]{amsart}

\usepackage[margin=1.05in]{geometry}
\usepackage{amsmath,amssymb,amsthm,mathtools}
\usepackage{microtype}
\usepackage{hyperref}

\hypersetup{
  colorlinks=true,
  linkcolor=blue,
  citecolor=blue,
  urlcolor=blue,
  pdftitle={Extended Tropical Characteristic Fibers and Slopes of Binomial D-Modules},
  pdfauthor={Jen-Chieh Hsiao}
}

\newtheorem{theorem}{Theorem}[section]
\newtheorem{proposition}[theorem]{Proposition}
\newtheorem{lemma}[theorem]{Lemma}
\newtheorem{corollary}[theorem]{Corollary}
\theoremstyle{definition}
\newtheorem{definition}[theorem]{Definition}
\theoremstyle{remark}
\newtheorem{remark}[theorem]{Remark}

\newcommand{\C}{\mathbb C}
\newcommand{\R}{\mathbb R}
\newcommand{\Z}{\mathbb Z}
\newcommand{\N}{\mathbb N}
\newcommand{\cT}{\mathcal T}
\newcommand{\gr}{\operatorname{gr}}
\newcommand{\Trop}{\operatorname{Trop}}
\newcommand{\Ch}{\operatorname{Ch}}
\newcommand{\Supp}{\operatorname{Supp}}
\newcommand{\Spec}{\operatorname{Spec}}
\newcommand{\Var}{\operatorname{Var}}
\newcommand{\qdeg}{\operatorname{qdeg}}
\newcommand{\im}{\operatorname{im}}
\newcommand{\conv}{\operatorname{conv}}
\newcommand{\Slope}{\operatorname{Slope}}
\newcommand{\Jump}{\operatorname{Jump}}
\newcommand{\ext}{\mathrm{ext}}
\newcommand{\red}{\mathrm{red}}
\newcommand{\TT}{\mathbb T}

\title[Extended tropical characteristic fibers]{Extended Tropical Characteristic Fibers and Slopes of Binomial $D$-Modules}
\author{Jen-Chieh Hsiao}
\address{Department of Mathematics, National Cheng Kung University, Tainan 701, Taiwan}
\email{jhsiao@mail.ncku.edu.tw}
\date{}


\begin{document}

\begin{abstract}
For a holonomic $D$-module, we consider the extended tropicalization of the reduced $L$-characteristic fiber along the family $L_s=F+(s-1)V_Y$. For $A$-hypergeometric systems, this tropical fiber determines the complete $(A,L)$-umbrella and hence recovers the classical slopes. We extend this result to holonomic binomial $D$-modules, for which the tropical jump parameters are exactly the classical slopes at the origin. As an application, a holonomic binomial $D$-module is regular holonomic if and only if there are no tropical jumps along any coordinate subspace.
\end{abstract}

\maketitle

\section{Introduction}

For an ordinary differential equation, the nonhorizontal slopes of the Newton polygon measure irregularity, and their absence is the Fuchsian regularity condition. In higher dimensions, slopes of holonomic $D$-modules are described through the variation of weighted characteristic varieties; see, for example, \cite{LaurentMebkhout,ACGSlopes}. For the family $L=pF+qV_Y$, where $F$ is the order filtration and $V_Y$ is the filtration along a smooth submanifold $Y$, the slopes occur at the exceptional parameters where the relevant $L$-characteristic components change.

Slopes along $Y$ are defined at a fixed point $y\in Y$. Accordingly, the relevant information comes from those components of $\Ch^L(M)$ that meet the cotangent fiber $T_y^*X$. This leads naturally to the reduced $L$-characteristic fiber
\[
  (\Ch^L(M)\cap T_y^*X)_{\red},
\]
and to the question of whether its variation remains visible after tropicalization. Ordinary tropicalization in the dense torus is insufficient, since components contained in coordinate hyperplanes of $T_y^*X$ would be lost. We therefore use extended tropicalization, which records all coordinate strata. As $L$ varies in the $F$--$V_Y$ family, these extended tropical characteristic fibers form a one-parameter family, and we study the parameters at which this family changes. Precise definitions are given in Sections 2 and 3.

Our model case is the $A$-hypergeometric system $M_A(\beta)$. Schulze and Walther describe its $L$-characteristic variety by conormal components indexed by the $(A,L)$-umbrella $\Phi_A^L$, whose changes as $L$ varies determine the slopes \cite{SchulzeWalther}. We show in Proposition~\ref{prop:gkz-fiber} that the intersection of each conormal component with $T_0^*X$ is scheme-theoretically the corresponding affine toric orbit closure. Standard toric tropical geometry then shows that the extended tropical characteristic fiber records exactly the faces of $\Phi_A^L$. Hence it determines the complete umbrella, thereby recovering the slopes of $M_A(\beta)$; see Theorem~\ref{thm:recover-umbrella} and Corollary~\ref{cor:gkz-slopes}.

The tropical detection of slopes established above for $A$-hypergeometric systems extends to holonomic binomial $D$-modules. The characteristic variety and slope set of a holonomic binomial $D$-module are unions of contributions from its active toral components \cite{FernandezCastro}. Proposition~\ref{prop:active-face-union} shows that the fixed-$L$ tropical characteristic fiber records the union of the corresponding active umbrella face sets, but not which active component contributes a given face. Lemma~\ref{lem:critical-face} shows that every change of an active umbrella produces a face whose tropical stratum occurs at the exceptional parameter but not at nearby parameters. This yields Theorem~\ref{thm:binomial-slopes}:
\[
  \Jump^{\mathrm{trop}}_{Y,0}(M_A(I,\beta))
  =
  \Slope_0(M_A(I,\beta),Y).
\]

Combining Theorem~\ref{thm:binomial-slopes} with the classical regularity criterion for holonomic binomial $D$-modules gives the main result.

\medskip
\noindent\textbf{Main Theorem (Theorem~\ref{thm:regularity}).}
Let $M_A(I,\beta)$ be a holonomic binomial $D$-module. Then $M_A(I,\beta)$ is regular holonomic if and only if, for every coordinate subspace $Y$,
\[
  \Jump^{\mathrm{trop}}_{Y,0}(M_A(I,\beta))=\varnothing.
\]
\medskip

This result may be viewed as a higher-dimensional analogue of the one-dimensional phenomenon observed in \cite[Example~2.6]{HsiaoTropD}, where tropicalization of a principal differential ideal recovers the nonhorizontal slopes of the Newton polygon and detects Fuchsian regularity. We return to this comparison in Section~\ref{sec:one-dimensional-comparison}.

\section{Weighted characteristic varieties, slopes, and tropicalization}\label{sec:background}

\subsection{The Weyl algebra and \texorpdfstring{$L$}{L}-characteristic varieties}

Let $X=\C^n$ and let
\[
  D=D_n(\C)=\C\langle x_1,\ldots,x_n,\partial_1,\ldots,\partial_n\rangle/([\partial_i,x_j]-\delta_{ij})
\]
be the $n$-th Weyl algebra. The order filtration is given by $F(x_i)=0$ and $F(\partial_i)=1$, so $\gr^F D\simeq\C[x,\partial]$ and $T^*X=\Spec\C[x,\partial]$. Following Schulze and Walther, we use the same symbols $\partial_i$ for the differential operators and their principal symbols; geometrically the latter are often denoted $\xi_i$.

More generally, a weight $L=(L_x,L_\partial)\in\R^{2n}$ with $L_{x_i}+L_{\partial_i}>0$ for every $i$ defines an increasing filtration on $D$ by maximum $L$-degree, with associated graded ring $\C[x,\partial]$. If $M$ is a coherent $D$-module with a good $L$-filtration, its $L$-characteristic variety is
\[
  \Ch^L(M)=\Supp_{\gr^L D}(\gr^L M)\subseteq T^*X.
\]
Its reduced support is independent of the chosen good filtration. For a cyclic module $M=D/I$, one has $\Ch^L(M)=\Var(\operatorname{in}_L(I))=\Var(\sqrt{\operatorname{in}_L(I)})$, so the slope problem below depends only on the reduced characteristic support.

\subsection{The \texorpdfstring{$F$--$V_Y$}{F--VY} family and slopes}

Let $Y=\Var(x_i:i\in V)\subseteq X$ be a coordinate subspace. The $V$-filtration along $Y$ has $V_Y(x_i)=-1$ and $V_Y(\partial_i)=1$ for $i\in V$, and weight zero on the remaining generators. We use the one-parameter family
\begin{equation}\label{eq:Ls}
  L_s=F+(s-1)V_Y,\qquad s>1.
\end{equation}
Thus $L_s(\partial_i)=s$ and $L_s(x_i)=1-s$ for $i\in V$, while $L_s(\partial_j)=1$ and $L_s(x_j)=0$ for $j\notin V$.

Fix $y\in Y$. We say that $s_0>1$ is a \emph{slope of $M$ along $Y$ at $y$} if the collection of irreducible components of $\Ch^{L_s}(M)$ that meet $T_y^*X$ is not locally constant at $s=s_0$; see \cite[Definition~3.4]{SchulzeWalther}. We denote the set of such slopes by
\[
  \Slope_y(M,Y).
\]
Throughout the paper, we express slope sets in the parameter $s$ defined by \eqref{eq:Ls}. In \cite{SchulzeWalther}, $L=pF+qV_Y$ is parametrized by $p/q$, which corresponds to $1/(s-1)$.

\subsection{Tropicalization and coordinate strata}

We recall only the tropical facts used below. Throughout, $\C$ carries the trivial valuation. For a closed subvariety $Z\subseteq(\C^*)^r$, write $\Trop(Z)\subseteq\R^r$ for its tropicalization; see \cite[Section~3.2]{MaclaganSturmfels}. If
\[
  \phi_B:(\C^*)^d\longrightarrow(\C^*)^r
\]
is the monomial map with exponent matrix $B$, then tropicalization commutes with this map \cite[Corollary~3.2.13]{MaclaganSturmfels}. In particular,
\[
  \Trop\!\left(\overline{\operatorname{im}\phi_B}\right)=\im(B^{\mathsf T}).
\]

For $\tau\subseteq[n]$, let $\TT_\tau\subseteq T_y^*X$ be the coordinate torus orbit whose nonzero coordinates are precisely those indexed by $\tau$, and let
\[
  N_\tau=\{w\in(\R\cup\{\infty\})^n:w_i<\infty\Longleftrightarrow i\in\tau\}
\]
be the corresponding tropical stratum. We identify $N_\tau$ with $\R^\tau$ by inserting $\infty$ in the complementary coordinates. For a closed subvariety $Z\subseteq T_y^*X$, its \emph{extended tropicalization} is defined orbitwise by
\begin{equation}\label{eq:extended-trop}
  \Trop^{\ext}(Z)\cap N_\tau:=\Trop(Z\cap\TT_\tau).
\end{equation}
Ordinary tropicalization records only the part of $Z$ lying in the dense torus, whereas extended tropicalization retains the information on all coordinate strata. It is compatible with finite unions.

We also use the usual orbit--face correspondence for affine toric varieties: the intersections of an affine toric orbit closure with the coordinate torus orbits are precisely the torus orbits corresponding to its faces; see \cite[Section~6.3]{MaclaganSturmfels}. Combining this with the monomial-map formula computes the extended tropicalization of an affine toric orbit closure stratum by stratum.

\section{Extended tropical characteristic fibers}\label{sec:construction}

\subsection{The fixed-\texorpdfstring{$L$}{L} construction}

Since slopes at $y$ are determined by the components of $\Ch^L(M)$ meeting $T_y^*X$, we refer to
\[
  (\Ch^L(M)\cap T_y^*X)_{\red}
\]
as the \emph{reduced $L$-characteristic fiber of $M$ at $y$}.

\begin{definition}\label{def:tropical-fiber}
Let $M$ be a holonomic $D$-module, let $Y$ be a coordinate subspace, let $y\in Y$, and let $L=pF+qV_Y$. The \emph{extended tropical characteristic fiber} is
\[
  \cT^L_{M,Y,y}
  :=
  \Trop^{\ext}\!\left((\Ch^L(M)\cap T_y^*X)_{\red}\right).
\]
\end{definition}
The filtration weight $L$ selects the characteristic variety, whereas a point of $\cT^L_{M,Y,y}$ is a tropical weight on the cotangent coordinates.

Along \eqref{eq:Ls}, a number $s_0>1$ is called a \emph{tropical jump parameter} if $s\mapsto\cT^{L_s}_{M,Y,y}$ is not locally constant at $s_0$, that is, if there is no neighborhood of $s_0$ on which this tropical characteristic fiber is constant. We denote the set of tropical jump parameters by
\[
  \Jump^{\mathrm{trop}}_{Y,y}(M).
\]

\subsection{Finite variation and slopes}

\begin{proposition}\label{prop:finite-variation}
For a coherent $D$-module $M$, the family $s\mapsto\cT^{L_s}_{M,Y,y}$ is locally constant away from finitely many rational values of $s$.
\end{proposition}

\begin{proof}
The ray $s\mapsto L_s$ meets only finitely many cones of the rational Gr\"obner fan of a filtered presentation of $M$ \cite{ACGGrFan}. On the relative interior of each resulting interval the associated graded presentation is constant. Hence its support, its intersection with the fixed cotangent fiber, the corresponding reduced variety, and its extended tropicalization are all constant.
\end{proof}

\begin{proposition}\label{prop:general-inclusion}
For every holonomic $D$-module $M$,
\[
  \Jump^{\mathrm{trop}}_{Y,y}(M)\subseteq\Slope_y(M,Y).
\]
\end{proposition}

\begin{proof}
If $s_0$ is not a slope, the irreducible components of $\Ch^{L_s}(M)$ meeting $T_y^*X$ are constant for $s$ near $s_0$. Hence their reduced union in the cotangent fiber, and therefore its extended tropicalization, is constant near $s_0$.
\end{proof}

\begin{remark}\label{rem:strict-general-inclusion}
The converse need not hold for an arbitrary holonomic $D$-module. Let $M'$ be a holonomic $D$-module having a nontrivial slope along $Y$ at the origin, produced by a characteristic component not contained in $T_0^*X$, and let
\[
  \delta_0=D/(Dx_1+\cdots+Dx_n)
\]
be the delta module at the origin. For every weight $L$ considered above,
\[
  \gr^L(\delta_0)\simeq\gr^L(D)/(x_1,\ldots,x_n),
  \qquad
  \Ch^L(\delta_0)=T_0^*X.
\]
Set $M=M'\oplus\delta_0$. Then $(\Ch^L(M)\cap T_0^*X)_{\red}=T_0^*X$ for every $L$, so the extended tropical characteristic fiber at the origin is constant. The slope-producing component of $M'$ remains a varying component of $\Ch^L(M)$, and the corresponding slope remains a slope of $M$. Thus the inclusion in Proposition~\ref{prop:general-inclusion} can be strict.
\end{remark}

\section{Characteristic fibers and toric geometry in the GKZ case}\label{sec:GKZ}

\subsection{GKZ systems and the \texorpdfstring{$(A,L)$}{(A,L)}-umbrella}

Let $A=(a_1,\ldots,a_n)\in\Z^{d\times n}$ have rank $d$ and assume that $\N A$ is pointed; after replacing the ambient lattice by $\Z A$, we may also assume $\Z A=\Z^d$. Let
\[
  I_A=\langle\partial^u-\partial^v:Au=Av\rangle\subseteq\C[\partial]
\]
be the toric ideal and $E_i=\sum_j a_{ij}x_j\partial_j$ the Euler operators. For $\beta\in\C^d$, the $A$-hypergeometric system is
\[
  M_A(\beta)=D/D(I_A,E_1-\beta_1,\ldots,E_d-\beta_d).
\]

The torus $T_A=(\C^*)^d$ acts on $T_0^*X=\Spec\C[\partial]$: for a point $z=(z_1,\ldots,z_n)$, $t\cdot z=(t^{a_1}z_1,\ldots,t^{a_n}z_n)$. For $\tau\subseteq[n]$, let $e_\tau$ be the point with coordinate $1$ on $\tau$ and $0$ outside $\tau$, and set $O_A^\tau=T_A\cdot e_\tau$. If $I_\tau$ is the toric ideal of $A_\tau$, then
\[
  I_A^\tau=\C[\partial]I_\tau+\langle\partial_j:j\notin\tau\rangle,
  \qquad
  \overline{O_A^\tau}=\Var(I_A^\tau)\simeq\Spec\C[\N A_\tau].
\]

Throughout this section, $L=(L_x,L_\partial)$ satisfies
\[
  L_x+L_\partial=c\mathbf 1\qquad(c>0),
\]
and $L(\partial_j)>0$ for every $j$. In particular, every $L_s$ in \eqref{eq:Ls} satisfies these conditions with $c=1$. Set
\[
  \Delta_A^L=\conv\!\left(0,\frac{a_1}{L(\partial_1)},\ldots,\frac{a_n}{L(\partial_n)}\right)\subset\R^d.
\]
The $(A,L)$-umbrella $\Phi_A^L$ is the face complex consisting of the faces of $\Delta_A^L$ that do not contain $0$, indexed by the corresponding subsets of $[n]$; see \cite[Section~2.3]{SchulzeWalther}. We include the empty face. For $\tau\in\Phi_A^L$, let $C_A^\tau$ denote the conormal space to $O_A^\tau$ in the notation of \cite[Definition~3.7]{SchulzeWalther}.

Schulze and Walther prove that, for every $\beta$,
\begin{equation}\label{eq:gkz-char}
  \Ch^L(M_A(\beta))=\bigcup_{\tau\in\Phi_A^L}\overline{C_A^\tau}
\end{equation}
\cite[Corollary~4.17]{SchulzeWalther}. The slopes along a coordinate subspace at the origin are precisely the parameters $s$ at which $\Phi_A^{L_s}$ changes \cite[Corollary~4.18]{SchulzeWalther}.

\subsection{From conormal components to toric orbit closures}

\begin{proposition}\label{prop:gkz-fiber}
For every $\tau\in\Phi_A^L$,
\[
  \overline{C_A^\tau}\cap T_0^*X=\overline{O_A^\tau}
\]
scheme-theoretically.
\end{proposition}

\begin{proof}
For every $z\in O_A^\tau$, the zero covector belongs to the conormal fiber over $z$. Thus $(0,z)\in C_A^\tau$, and hence $\overline{O_A^\tau}$ is contained in the intersection. We show that this is the whole intersection scheme-theoretically.

Let $P_\tau\subset S=\C[x,\partial]$ be the prime ideal of $\overline{C_A^\tau}$. Since $T_0^*X$ is defined by $x=0$, the scheme-theoretic intersection is obtained by restricting $P_\tau$ to $x=0$. Let
\[
  \varepsilon:S\longrightarrow\C[\partial],\qquad f(x,\partial)\longmapsto f(0,\partial)
\]
be the evaluation map. It is enough to show that $\varepsilon(P_\tau)=I_A^\tau$.

If $f\in P_\tau$, then $f(0,z)=0$ for every $z\in O_A^\tau$, so $\varepsilon(f)\in I_A^\tau$. Hence $\varepsilon(P_\tau)\subseteq I_A^\tau$. Conversely, if $g\in I_A^\tau$, regard $g$ as a polynomial independent of $x$. Since $g$ vanishes on $O_A^\tau$, it vanishes on every conormal fiber over $O_A^\tau$, and hence on $\overline{C_A^\tau}$. Thus $g\in P_\tau$, giving the reverse inclusion. Therefore $\varepsilon(P_\tau)=I_A^\tau$, and the result follows.
\end{proof}

Applying Proposition~\ref{prop:gkz-fiber} to the conormal components in \eqref{eq:gkz-char} gives
\begin{equation}\label{eq:gkz-reduced-fiber}
  (\Ch^L(M_A(\beta))\cap T_0^*X)_{\red}
  =\bigcup_{\tau\in\Phi_A^L}\overline{O_A^\tau}.
\end{equation}
Thus the reduced characteristic fiber is obtained from the $(A,L)$-umbrella by replacing each face with the corresponding affine toric orbit closure.

\subsection{Tropicalization and recovery of the umbrella}

Let $N_\sigma$ be the tropical coordinate stratum associated to $\sigma\subseteq[n]$. By the orbit--face correspondence recalled in Section 2.3, $\overline{O_A^\tau}$ meets $\TT_\sigma$ if and only if $\sigma$ is a face of $\tau$, and in that case the intersection is $O_A^\sigma$. Applying the monomial-map formula on this stratum gives
\begin{equation}\label{eq:orbit-trop}
  \Trop^{\ext}(\overline{O_A^\tau})\cap N_\sigma
  =\im(A_\sigma^{\mathsf T})\times\{\infty\}^{\sigma^c}
\end{equation}
when $\sigma$ is a face of $\tau$, and the intersection is empty otherwise.

For brevity write
\[
  \cT_A^L
  :=\Trop^{\ext}\!\left((\Ch^L(M_A(\beta))\cap T_0^*X)_{\red}\right).
\]
Its reduced support, and hence $\cT_A^L$, is independent of $\beta$.

\begin{theorem}\label{thm:recover-umbrella}
For every $\sigma\subseteq[n]$,
\[
  \cT_A^L\cap N_\sigma
  =
  \begin{cases}
    \im(A_\sigma^{\mathsf T})\times\{\infty\}^{\sigma^c},&\sigma\in\Phi_A^L,\\[1mm]
    \varnothing,&\sigma\notin\Phi_A^L.
  \end{cases}
\]
Hence $\cT_A^L$ determines the complete $(A,L)$-umbrella.
\end{theorem}

\begin{proof}
By \eqref{eq:gkz-reduced-fiber}, the reduced characteristic fiber is the union of $\overline{O_A^\tau}$ over $\tau\in\Phi_A^L$. Formula \eqref{eq:orbit-trop} shows that $N_\sigma$ occurs precisely when $\sigma$ is a face of some $\tau\in\Phi_A^L$. Since $\Phi_A^L$ is a face complex, this is equivalent to $\sigma\in\Phi_A^L$. On $N_\sigma$, every containing face contributes the same linear space $\im(A_\sigma^{\mathsf T})$.
\end{proof}

\begin{corollary}\label{cor:gkz-slopes}
For every coordinate subspace $Y$,
\[
  \Jump^{\mathrm{trop}}_{Y,0}(M_A(\beta))
  =\Slope_0(M_A(\beta),Y).
\]
\end{corollary}

\begin{proof}
By Theorem~\ref{thm:recover-umbrella}, the family $\cT_A^{L_s}$ is locally constant exactly when $\Phi_A^{L_s}$ is locally constant, and the latter changes precisely at the slopes \cite[Corollary~4.18]{SchulzeWalther}.
\end{proof}

\subsection{An elementary example}\label{sec:elementary-example}

Let $A=(1,m)$ with $m\ge2$, and consider
\[
  M_A(\beta)=D_2/D_2(\partial_1^m-\partial_2,\,x_1\partial_1+mx_2\partial_2-\beta).
\]
Take $Y=\{x_2=0\}$. Under $L_s$, the differential weights are $1$ and $s$, so the scaled columns are $1$ and $m/s$. Apart from the empty face,
\[
  \Phi_A^{L_s}=
  \begin{cases}
    \{\{2\}\},&1<s<m,\\
    \{\{1,2\}\},&s=m,\\
    \{\{1\}\},&s>m.
  \end{cases}
\]
Thus the unique slope is $s=m$. The corresponding reduced characteristic fibers are
\[
  (\Ch^{L_s}(M_A(\beta))\cap T_0^*X)_{\red}
  =
  \begin{cases}
    \Var(\partial_1),&s<m,\\
    \Var(\partial_1^m-\partial_2),&s=m,\\
    \Var(\partial_2),&s>m,
  \end{cases}
\]
and, with tropical coordinates $(w_1,w_2)$,
\[
  \cT_A^{L_s}=
  \begin{cases}
    \{(\infty,u):u\in\R\}\cup\{(\infty,\infty)\},&s<m,\\
    \{(u,mu):u\in\R\}\cup\{(\infty,\infty)\},&s=m,\\
    \{(u,\infty):u\in\R\}\cup\{(\infty,\infty)\},&s>m.
  \end{cases}
\]
Ordinary tropicalization in the dense torus sees only the middle case, whereas extended tropicalization also records the two coordinate-stratum cases.

\section{Holonomic binomial \texorpdfstring{$D$}{D}-modules}\label{sec:binomial}

\subsection{Active toral components}

Let $R=\C[\partial_1,\ldots,\partial_n]$ be graded by $\deg(\partial_j)=-a_j$, let $I\subset R$ be an $A$-graded binomial ideal, and let
\[
  M_A(I,\beta)=D/H_A(I,\beta),\qquad H_A(I,\beta)=DI+\sum_{i=1}^dD(E_i-\beta_i),
\]
be the associated binomial $D$-module. Fix a binomial primary decomposition $I=\bigcap_{\rho,J}C_{\rho,J}$, with $C_{\rho,J}$ primary to
\[
  I_{\rho,J}=I_\rho+\langle\partial_j:j\notin J\rangle,
\]
where $I_\rho\subset\C[\partial_j:j\in J]$ is a lattice ideal determined by a partial character; see \cite{DMM}. Write $A_J=(a_j)_{j\in J}$ and
\[
  D_J=\C\langle x_j,\partial_j:j\in J\rangle.
\]
The prime $I_{\rho,J}$ is \emph{toral} when $\dim(R/I_{\rho,J})=\operatorname{rank}(A_J)$, equivalently when its $A$-graded Hilbert function is bounded \cite{DMM}.

For a finitely generated $A$-graded $R$-module $Q$, its quasidegree set $\qdeg(Q)\subset\C^d$ is the Zariski closure of the set of degrees in which $Q$ is nonzero. For brevity, we call a toral primary component $C_{\rho,J}$ \emph{active for $\beta$} if
\[
  -\beta\in\qdeg(R/C_{\rho,J}),
\]
and use the same shorthand for its associated toral prime $I_{\rho,J}$.

Assume throughout this section that $M_A(I,\beta)$ is holonomic. For every $L$ satisfying $L_x+L_\partial=c\mathbf 1$ with $c>0$ and $L(\partial_j)>0$ for all $j$, Fernández-Fernández and Castro-Jiménez prove
\begin{align}
  \Ch^L(M_A(I,\beta))
  &=\bigcup_{(\rho,J)\ \mathrm{active}}\Ch^L(M_A(I_{\rho,J},0)),\label{eq:binomial-char-union}\\
  \Slope_0(M_A(I,\beta),Y)
  &=\bigcup_{(\rho,J)\ \mathrm{active}}\Slope_0(M_A(I_{\rho,J},0),Y)\label{eq:binomial-slope-union}
\end{align}
for every coordinate subspace $Y$; see \cite[Theorem~4.3]{FernandezCastro}.

For a toral prime $I_{\rho,J}$, the corresponding system is hypergeometric in the $J$-coordinates: after rescaling these coordinates by nonzero constants determined by $\rho$, its $D_J$-part is the GKZ system for $A_J$ \cite[Remark~3.4]{FernandezCastro}, while $\partial_j=0$ for $j\notin J$. Since the ground field is trivially valued, this rescaling does not affect the tropicalizations considered here.

If $A_J$ does not have full row rank, we regard it as a full-rank configuration in its real linear span. Equivalently, one may replace it by a matrix with the same rational row space and full row rank. This leaves its toric ideal and the spaces $\im(A_\sigma^{\mathsf T})$ unchanged and identifies the corresponding umbrellas by a linear isomorphism of the column span. Applying Section 4 to each active contribution gives, for $\sigma\subseteq[n]$,
\begin{equation}\label{eq:binomial-fixed-L}
  \cT^L_{M_A(I,\beta),Y,0}\cap N_\sigma
  =
  \begin{cases}
    \im(A_\sigma^{\mathsf T})\times\{\infty\}^{\sigma^c},
    &\displaystyle \sigma\in\bigcup_{(\rho,J)\ \mathrm{active}}\Phi_{A_J}^{L|_J},\\[2mm]
    \varnothing,&\text{otherwise}.
  \end{cases}
\end{equation}
Here $L|_J$ denotes the restriction of the differential weights to $J$; in particular, every face of $\Phi_{A_J}^{L|_J}$ is regarded as a subset of $J\subseteq[n]$.

\begin{proposition}\label{prop:active-face-union}
For fixed $L$ as above, the extended tropical characteristic fiber determines exactly the union
\[
  \bigcup_{(\rho,J)\ \mathrm{active}}\Phi_{A_J}^{L|_J}
\]
of the active umbrella face sets. It records whether a face occurs in this union, but not which active toral component contributes it.
\end{proposition}

\begin{proof}
By \eqref{eq:binomial-fixed-L}, the stratum $N_\sigma$ occurs precisely when $\sigma$ belongs to this union, and whenever it occurs its finite-coordinate part is $\im(A_\sigma^{\mathsf T})$, which depends only on $A_\sigma$.
\end{proof}

\subsection{Critical faces}

For a configuration $B=(b_j)_{j\in K}$ indexed by a finite set $K$ and a positive weight vector $\lambda=(\lambda_j)_{j\in K}$, let $\Phi_B^\lambda$ denote the face complex obtained from
\[
  \conv\!\left(0,\frac{b_j}{\lambda_j}:j\in K\right)
\]
by taking the faces not containing the origin. For $\tau\subseteq K$, write $\lambda_\tau=(\lambda_j)_{j\in\tau}$. Along the slope family we use
\[
  \lambda_\tau(s)=(L_s(\partial_j))_{j\in\tau}.
\]
If an umbrella $\Phi_{A_J}^{L_s|_J}$ changes at $s_0$, we call a face $\tau\in\Phi_{A_J}^{L_{s_0}|_J}$ \emph{critical at $s_0$} if
\[
  \lambda_\tau(s_0)\in\im(A_\tau^{\mathsf T}),
  \qquad
  \lambda_\tau(s)\notin\im(A_\tau^{\mathsf T})\quad(0<|s-s_0|\ll1).
\]
The first condition says that the scaled columns indexed by $\tau$ lie on a common supporting hyperplane at $s_0$; the second says that these equalities cannot hold simultaneously at nearby parameters.

\begin{lemma}\label{lem:critical-face}
Let $B=(b_j)_{j\in K}$ be a pointed configuration of rank $r$, regarded in its real span $\R B$, and let
\[
  \lambda(s)=(\lambda_j(s))_{j\in K}
\]
be an affine one-parameter family of positive weights. If $\Phi_B^{\lambda(s)}$ is not locally constant at $s_0$, then there is a facet $\tau\in\Phi_B^{\lambda(s_0)}$ such that
\[
  \lambda_\tau(s_0)\in\im(B_\tau^{\mathsf T}),
  \qquad
  \lambda_\tau(s)\notin\im(B_\tau^{\mathsf T})\quad(0<|s-s_0|\ll1).
\]
\end{lemma}

\begin{proof}
After shrinking around $s_0$, the umbrella is constant on each of the two punctured intervals. Since its face complex is determined by its facets, after reversing the parameter if necessary we may choose a facet $\sigma$ for $s<s_0$ such that at least one of its strict supporting inequalities becomes an equality at $s_0$.

Because $\sigma$ is a facet not containing the origin, the columns $B_\sigma$ span $\R B$. Its supporting functional $h_s\in(\R B)^*$ is therefore uniquely determined by
\[
  h_s(b_j)=\lambda_j(s)\ (j\in\sigma),
  \qquad
  h_s(b_k)<\lambda_k(s)\ (k\notin\sigma),
\]
and depends affinely on $s$. Put $d_k(s)=\lambda_k(s)-h_s(b_k)$ and
\[
  \tau=\{k:d_k(s_0)=0\}.
\]
For every $j\in\sigma$, the affine function $d_j$ vanishes for $s<s_0$ and hence identically near $s_0$, so $\sigma\subseteq\tau$. Moreover, the inequalities $d_k(s)>0$ for $k\notin\sigma$ and $s<s_0$ imply $d_k(s_0)\geq0$. Thus $h_{s_0}$ is supporting and $\tau$ is the full index set of the face cut out by its supporting hyperplane.

The scaled points indexed by $\sigma$ still span $\R B$ linearly and lie on this hyperplane. They therefore affinely span the hyperplane, so the face indexed by $\tau$ is a facet of $\Phi_B^{\lambda(s_0)}$. The index set $\sigma$ need not define a face at $s_0$; it is contained in the larger facet $\tau$. By the choice of $\sigma$, at least one strict inequality becomes equality at $s_0$, and hence $\sigma\subsetneq\tau$.

Choose $k\in\tau\setminus\sigma$. The affine function $d_k$ is positive for $s<s_0$ sufficiently close to $s_0$ and vanishes at $s_0$. Since it is not identically zero, $d_k(s)\neq0$ for all sufficiently close $s\neq s_0$.

At $s_0$, the functional $h_{s_0}$ gives $\lambda_\tau(s_0)\in\im(B_\tau^{\mathsf T})$. Suppose that $\lambda_\tau(s)\in\im(B_\tau^{\mathsf T})$ for some sufficiently close $s\neq s_0$. Then there is $g_s\in(\R B)^*$ with $g_s(b_j)=\lambda_j(s)$ for all $j\in\tau$. On $\sigma$, the functionals $g_s$ and $h_s$ satisfy the same equalities. Since $B_\sigma$ spans $\R B$, they coincide. Evaluating at the chosen $k$ gives $d_k(s)=0$, a contradiction.
\end{proof}

\subsection{Slope detection for binomial modules}

\begin{theorem}\label{thm:binomial-slopes}
Let $M_A(I,\beta)$ be a holonomic binomial $D$-module and let $Y$ be a coordinate subspace. Then
\[
  \Jump^{\mathrm{trop}}_{Y,0}(M_A(I,\beta))
  =\Slope_0(M_A(I,\beta),Y).
\]
\end{theorem}

\begin{proof}
The inclusion from left to right is Proposition~\ref{prop:general-inclusion}. Conversely, let $s_0\in\Slope_0(M_A(I,\beta),Y)$. By \eqref{eq:binomial-slope-union}, there is an active toral component $C_{\rho,J}$ such that $\Phi_{A_J}^{L_s|_J}$ changes at $s_0$. Apply Lemma~\ref{lem:critical-face} to $B=A_J$, and let $\tau\subseteq J$ be the resulting critical facet.

At $s=s_0$, \eqref{eq:binomial-fixed-L} gives
\[
  \cT^{L_{s_0}}_{M_A(I,\beta),Y,0}\cap N_\tau\neq\varnothing.
\]
If the same stratum occurred for some sufficiently close $s\neq s_0$, then \eqref{eq:binomial-fixed-L} would place $\tau$ in the umbrella of some active toral component. A supporting functional for that face would imply $\lambda_\tau(s)\in\im(A_\tau^{\mathsf T})$, contrary to criticality. Hence the stratum is absent for all sufficiently close $s\neq s_0$, and the tropical characteristic family jumps at $s_0$.
\end{proof}

\section{Tropical characterization of regularity}\label{sec:regularity}

For GKZ systems and holonomic binomial $D$-modules, the preceding sections recover the classical slopes from the variation of the extended tropical characteristic fiber. We now combine this with the classical relation between slopes and regular holonomicity.

\subsection{Regularity and homogeneity}

Fernández-Fernández and Castro-Jiménez prove that a holonomic binomial module $M_A(I,\beta)$ is regular holonomic if and only if every active toral associated prime is homogeneous \cite[Theorem~3.7]{FernandezCastro}. For a toral prime $I_{\rho,J}$, this is equivalent to homogeneity of the configuration $A_J$ \cite[Lemma~3.5]{FernandezCastro}: there is a rational linear functional $h:\mathbb Q A_J\to\mathbb Q$ with $h(a_j)=1$ for all $j\in J$, or equivalently the columns of $A_J$ lie on an affine hyperplane not containing the origin.

If an active $A_J$ is nonhomogeneous, the associated hypergeometric system has a slope along a suitable coordinate subspace at the origin \cite[Corollary~3.16]{SchulzeWalther}. By \eqref{eq:binomial-slope-union}, the same slope occurs for $M_A(I,\beta)$.

\subsection{The tropical regularity criterion}

\begin{theorem}\label{thm:regularity}
Let $M_A(I,\beta)$ be a holonomic binomial $D$-module. Then $M_A(I,\beta)$ is regular holonomic if and only if, for every coordinate subspace $Y$,
\[
  \Jump^{\mathrm{trop}}_{Y,0}(M_A(I,\beta))=\varnothing.
\]
\end{theorem}

\begin{proof}
If $M_A(I,\beta)$ is regular holonomic, it has no slopes along coordinate subspaces. By Theorem~\ref{thm:binomial-slopes}, there are no tropical jumps.

Conversely, suppose that $M_A(I,\beta)$ is irregular. Then some active toral prime $I_{\rho,J}$ has nonhomogeneous configuration $A_J$. As recalled above, its hypergeometric system has a slope along a suitable coordinate subspace at the origin. Formula \eqref{eq:binomial-slope-union} makes this a slope of $M_A(I,\beta)$, and Theorem~\ref{thm:binomial-slopes} gives a tropical jump. Thus the stated vanishing condition fails.
\end{proof}

\subsection{A parameter-dependent binomial example}

We reinterpret \cite[Example~4.5]{FernandezCastro}. Let
\[
  A=\begin{pmatrix}1&0&1&2&3\\0&1&1&2&3\end{pmatrix}
\]
and let $I$ be the binomial ideal considered there. Its toral prime supports are $J_1=\{3,5\}$, $J_2=\{4,5\}$, $J_3=\{1,2,5\}$, and $J_4=\{5\}$. Writing $I_i$ for the corresponding toral prime, the partial characters are trivial, $\qdeg(R/I_i)=\C\cdot(1,1)$ for $i=1,2,4$, and $\qdeg(R/I_3)=\C^2$.

For $Y=\{x_5=0\}$, the differential weights are $(1,1,1,1,s)$. The relations $\partial_4^3-\partial_5^2$, $\partial_3^3-\partial_5$, and $\partial_1^3\partial_2^3-\partial_5$ give the exceptional parameters $s=3/2,3,6$, with tropical strata $N_{\{4,5\}}$, $N_{\{3,5\}}$, and $N_{\{1,2,5\}}$. Their finite-coordinate parts are, respectively,
\[
  \{(2u,3u):u\in\R\},\qquad
  \{(u,3u):u\in\R\},\qquad
  \{(u,v,3u+3v):u,v\in\R\}.
\]
If $\beta\notin\C\cdot(1,1)$, only $J_3$ contributes and
\[
  \Jump^{\mathrm{trop}}_{Y,0}(M_A(I,\beta))=\{6\}.
\]
If $\beta\in\C\cdot(1,1)$, all three slope-producing components contribute and
\[
  \Jump^{\mathrm{trop}}_{Y,0}(M_A(I,\beta))=\{3/2,3,6\}.
\]
These are exactly the slope sets in \cite[Example~4.5]{FernandezCastro}.

\subsection{The one-dimensional picture}\label{sec:one-dimensional-comparison}

We conclude by returning to one variable, which provides the basic model for the preceding construction. Let
\[
  P=\sum_{i,j}c_{ij}x^i\partial^j
\]
be a differential operator. For $Y=\{x=0\}$,
\[
  L_s(x^i\partial^j)=j+(s-1)(j-i).
\]
Thus comparison of the $L_s$-degrees of the monomials $x^i\partial^j$ is the familiar supporting-line calculation for the Newton polygon of $P$. As $s$ varies, the corresponding initial data change only at the parameters determined by its nonhorizontal edges; these are the classical slopes.

The direct tropicalization of differential ideals studied in \cite{HsiaoTropD} makes this relation explicit. In one variable, \cite[Example~2.6]{HsiaoTropD} shows that the tropicalization of the principal ideal $DP$ recovers the nonhorizontal Newton-polygon slopes, and that its emptiness is equivalent to Fuchsian regularity. The construction in the present paper is different: an $L_s$-weight first determines the $L_s$-characteristic variety, which is then restricted to a cotangent fiber and tropicalized orbitwise. Sections 4 and 5 show that, for GKZ systems and holonomic binomial $D$-modules, this intermediate geometric step still retains all slope information.

Thus Theorem~\ref{thm:regularity} may be viewed as a higher-dimensional counterpart of the one-dimensional picture. For an ordinary differential equation, regular singularity is equivalent to the absence of nontrivial Newton-polygon slopes; for a holonomic binomial $D$-module, regular holonomicity is equivalent to the absence of tropical jumps along all coordinate subspaces.

\end{document}